\documentclass{amsart}

\usepackage{mabliautoref}
\usepackage{joe-custom}

\providecommand{\rddown}[1]{\left\lfloor #1\right\rfloor}

\begin{document}

\title{The cone theorem for effective fourfold pairs in characteristic $p>5$}
\author{Joe Waldron}
\address{Department of Mathematics, Michigan State University, East Lansing, MI 48824, USA}
\email{waldro51@msu.edu}

\begin{abstract}
Assuming the log resolution conjecture for all log pairs birational to $X$, we
prove the cone theorem for projective log canonical,
$\mathbb{Q}$-factorial fourfold pairs $(X,\Delta)$ such that $K_X+\Delta\equiv M\geq 0$ over bases of positive and mixed characteristic $p>5$. 
\end{abstract}

\maketitle

\section{Introduction}

The log minimal model program is now known in broad generality for threefolds.
Over algebraically closed fields of characteristic $p>5$, the theory was
developed in
\cite{hacon_three_2015,cascini_base_2015,xu_base-point-free_2015,
birkar_existence_2016,birkar_existence_2017,waldron_lmmp_2017}, and it was
extended to log canonical pairs over perfect fields and then to arbitrary
fields in \cite{hashizume_nakamura_tanaka,das_waldron_imperfect,fever_dream}.
Relative results in
arbitrary positive characteristic and the full characteristic-five program are
proved in \cite{hacon_low,hacon_five}.  In mixed characteristic, see
\cite{takamatsu_yoshikawa,characteristic_five,BMPSTWW1}.  One of the earliest
positive characteristic contributions was Keel's proof of the cone and
basepoint free theorems for threefolds
with $K_X+\Delta\equiv M\geq 0$.  The purpose of this note is to prove a 
four-dimensional analogue of Keel's version of the cone theorem
\cite[Theorem 0.6]{keel_basepoint_1999}.

The LMMP is already partially known in dimension four \cite{hacon_four, xie_xue_fourfolds}: in particular the
relative minimal model program can be run over a $\mathbb{Q}$-factorial
fourfold so long as the exceptional locus is contained in the reduced
boundary, contingent on the existence of log resolutions.  This relative MMP forms a key part of our proof.

\begin{theorem}\label{thm:main-cone}
Let $R$ be a finite-dimensional excellent ring admitting a dualizing complex, with no primes with residue characteristics $2$, $3$, or $5$,
and let $U$ be quasi-projective over $\operatorname{Spec}R$.  Let $(X,\Delta)$ be a  $\mathbb{Q}$-factorial log canonical pair of dimension four,
projective over $U$, with $\Delta$ a $\mathbb{Q}$-boundary.  Suppose that the
log-resolution \autoref{conj:res} holds for $X$ and that
\[
  K_X+\Delta\equiv_U M
\]
for an effective $\mathbb{Q}$-divisor $M$.  Then there are countably many
$(K_X+\Delta)$-negative extremal rays $R_i$ such that
\begin{enumerate}
\item
\[
 \overline{NE}(X/U)=
 \overline{NE}(X/U)_{K_X+\Delta\geq 0}+\sum_i R_i.
\]
\item The rays $R_i$ do not accumulate in
$\overline{NE}(X/U)_{K_X+\Delta<0}$.
\item Each $R_i$ is generated by an integral curve $C_i$ over $U$ which is the image of a rational curve in the sense of \autoref{def:geometric-rational-image} and
satisfies
\[
 -6d_{C_i}\leq (K_X+\Delta)\cdot C_i<0.
\]

\end{enumerate}
Moreover, for every $U$-ample $\mathbb{Q}$-divisor $A$, only finitely many of the
rays $R_i$ are $(K_X+\Delta+A)$-negative.
\end{theorem}

We also obtain some immediate corollaries, namely the rationality of nef thresholds and the existence of pl contractions, the latter of which was already reduced to the cone theorem by \cite[Proposition 4.5]{hacon_four}:

\begin{corollary}
\label{cor:rationality}
In the setting of \autoref{thm:main-cone}, suppose that $K_X+\Delta$ is not
nef over $U$.  Let $H$ be a $U$-ample Cartier divisor, and let $a$ be a
positive integer such that $a(K_X+\Delta)$ is Cartier.  Then
\[
 r:=\sup\{t\in\mathbb R_{\geq0}\mid
              H+t(K_X+\Delta)\text{ is nef over }U\}
\]
is a positive rational number.  More precisely, if $r=u/v$ in lowest terms,
then $0<v\leq6a$.  
\end{corollary}

\begin{corollary}\label{cor:pl-contraction}
In the setting of \autoref{thm:main-cone}, suppose in addition that
$(X,\Delta)$ is dlt.  Let $R\subset \overline{NE}(X/U)$ be a
$(K_X+\Delta)$-negative extremal ray, and suppose that
\(
 S\cdot R<0
\)
for some irreducible component $S$ of $\lfloor\Delta\rfloor$.  Then there is a
normal scheme $Z$, projective over $U$, and a projective contraction over $U$
\[
 f\colon X\longrightarrow Z
\]
such that $\rho(X/Z)=1$, and an irreducible
curve $C$ over $U$ is contracted if and only if $[C]\in R$.
\end{corollary}

\subsection{Outline of the proof}
As in Keel's proof, given a $(K_X+\Delta)$-negative extremal ray $R$, the aim is
to find a curve in $R$ satisfying the length bound in \autoref{thm:main-cone}(3).  The argument starts
by noting that $R\cdot S<0$ for some component $S$ of $M$.  If $S$ is also a
component of $\rddown{\Delta}$, then we are done by applying adjunction and the
cone theorem in dimension three.  In Keel's threefold situation this could be forced by adding enough of $M$ to raise the coefficient of $S$ to $1$,
while ignoring the resulting singularities, because the cone theorem applies
on surfaces very generally.  For us this is not good enough, because we need
to ensure that our threefold remains dlt after restriction.  So the aim of the
proof is to use suitable birational modifications to reduce to the case where
our extremal ray is negative for a component of $\rddown{\Delta}$.

To do so, we may assume that $R\cdot T\geq 0$ for every component $T$ of
 $\rddown{\Delta}$, otherwise we are done.  This means we may replace $\Delta$ by $\Delta-\epsilon \rddown{\Delta}$ to assume that $(X,\Delta)$ is klt, and also fix a component $S$ of $M$ such that $R\cdot S<0$.  
  We have that
$\lambda=\lct(X,\Delta,S)>0$, and let $X'$ be a dlt modification of
$(X,\Delta+\lambda S)$.  We fix an extremal ray $R'$ on $X'$ dominating $R$,
and may continue the process by replacing $X$ and $R$ by $X'$ and $R'$.  If we do not eventually reach a situation where
$R\cdot T<0$ for $T$ a component of $\rddown{\Delta}$, then we generate an
infinite increasing chain of log canonical thresholds. Unfortunately, the ascending chain condition for
log canonical thresholds is not known in dimension four.  However, we do know that it holds at
codimension three points by \cite[Theorem 1.5]{fever_dream}.  Therefore
instead of taking log canonical thresholds, we use log canonical thresholds
in codimension three.  This results in the pullbacks not being dlt, and some
divisors over closed points pick up coefficients greater than one.  However, the lifted extremal rays are not contracted by the
birational maps, so the resulting divisors have nonnegative
intersection with them, so the resulting inequality works in our favor.

\begin{remark}
	All occurrences of the number $6=2\cdot(\dim(X)-1)$ in bounds on lengths of extremal rays in this paper can be replaced by the
	better bound $4=(\dim(X)-1)+1$ by substituting the optimal bend-and-break result
	\cite[Theorem 1.1]{riedl_bend_and_break} into the proofs of the cited
	threefold cone theorems.  We have refrained from doing so explicitly because
	the original references use the original, weaker bound \cite{kollar_rational_1996}.
\end{remark}

\subsection*{Acknowledgments}
Waldron was supported by NSF CAREER Grant DMS \#2440240, NSF Grant DMS \#2401279 and the Simons Foundation Gift ID \#850684.  He would like to thank Zsolt Patakfalvi and Jakub Witaszek for helpful conversations related to this work. 

\subsection*{AI Statement}
I worked out the outline of this proof with all essential ingredients, in Spring 2025.  Unfortunately, there remained one gap that I could not repair.  I gave the unfinished proof to ChatGPT 5.6 Sol in Summer 2026, asking it to fill the gap.  It modified the approach to avoid the issue, which after further editing by hand and with Codex, resulted in the current version. 

\section{Preliminaries}

Throughout the paper, our base scheme $U$ will be quasi-projective over an excellent ring, which is of finite Krull dimension, admits a dualizing complex, and has no residue fields of characteristic $2$, $3$ or $5$.  Dimension will always refer to the dimension of the scheme, not relative dimension. 
We use standard terminology around the LMMP and related singularity classes, see \cite{kollar_birational_1998, kollar_singularities_2013}, and \cite{BMPSTWW1} for more details in our context.

\subsection{Curves and the cone of curves}

Unless another base is displayed, numerical equivalence, nefness, ampleness,
bigness, and semiampleness are relative to $U$.  A \emph{curve over $U$} is a
proper integral curve $C\subset X$ whose image is a closed point $u\in U$.
If $L$ is a Cartier divisor on $X$, we write
\[
 L\cdot C:=\deg_{\kappa(u)}(L|_C).
\]
Let $N_1(X/U)$ be the real vector space generated by curves over $U$, modulo
numerical equivalence against Cartier divisors on $X$, and let
$\overline{NE}(X/U)$ be the closure of the cone generated by their classes.

\begin{definition}
Let $C\subset X$ be a curve over $U$, let $u\in U$ be its image, and set
\[
 X_{\overline u}:=X\times_U\operatorname{Spec}\overline{\kappa(u)}.
\]
Let $\varphi_u:X_{\overline u}\to X$ be the projection.  If $\overline C$ is
an integral component of
$(C_{\overline{\kappa(u)}})_{\mathrm{red}}$, there is a unique positive
integer $d_C$, independent of the choice of $\overline{C}$, such that
\begin{equation}\label{eq:curve-index}
 L\cdot C=d_C\bigl(\varphi_u^*L\cdot_{\overline{\kappa(u)}}\overline C\bigr)
\end{equation}
for every Cartier divisor $L$ on $X$; see
\cite[Lemma~2.47]{BMPSTWW1} and
\cite[Lemmas~3.6 and 4.1]{das_waldron_imperfect}.  In particular,
$L\cdot C$ is divisible by $d_C$.  We call $d_C$ the \emph{curve index} and
write
\[
 \langle C\rangle:=\frac{[C]}{d_C}\in N_1(X/U)
\]
for the normalized numerical class.
\end{definition}

\begin{definition}\label{def:geometric-rational-image}
An integral curve $C\subset X$ over a closed point $u\in U$ is a
\emph{geometric rational image} if there is a nonconstant morphism
\(
 \mathbb P^1_{\overline{\kappa(u)}}
 \longrightarrow (X_{\overline u})_{\mathrm{red}}
\)
whose composite with $(X_{\overline u})_{\mathrm{red}}\to X$ has 
image $C$.
\end{definition}

\begin{lemma}\label{lem:curve-index-pushforward}
Let $g:Y\to X$ be a proper morphism over $U$, let $D\subset Y$ be a curve
over $U$ which is not contracted, and let
$C=g(D)_{\mathrm{red}}$.
If $g_*D=mC$, then
\(
 d_D\leq m d_C.
\)
\end{lemma}

\begin{proof}
The two curves lie over the same closed point $u\in U$.  Choose compatible irreducible components $\overline{C}$ and $\overline{D}$ of 
$(C_{\overline{\kappa(u)}})_{\mathrm{red}}$ and 
$(D_{\overline{\kappa(u)}})_{\mathrm{red}}$ for which the induced map
$\overline g:\overline D\to\overline C$ has degree $e\geq1$.  Let $H$ be a
$U$-ample Cartier divisor on $X$, and let $\varphi_X:X_{\overline u}\to X$ and
$\varphi_Y:Y_{\overline u}\to Y$ be the projections.  Equation
\eqref{eq:curve-index} and the projection formula give
\[
\begin{aligned}
 m d_C(\varphi_X^*H\cdot\overline C)
 &=m(H\cdot C)
  =H\cdot g_*D
  =g^*H\cdot D\\
 &=d_D(\varphi_Y^*g^*H\cdot\overline D)
  =e d_D(\varphi_X^*H\cdot\overline C).
\end{aligned}
\]
Since $\varphi_X^*H\cdot\overline C\neq 0$, we have
$m d_C=e d_D$ and the inequality follows.  
\end{proof}

\begin{lemma}
	\label{lem:normalized-classes-discrete}
	Let $X\to U$ be projective and let $H$ be a $U$-ample Cartier divisor.  For
	every $b>0$, the set
	\[
	\left\{\langle C\rangle\ \middle|\
	C\text{ is a curve over }U,\quad H\cdot\langle C\rangle\leq b\right\}
	\]
	is finite.
\end{lemma}

\begin{proof}
	Choose Cartier divisors
	\(
	L_1,\ldots,L_\rho
	\)
	whose numerical classes form an $\mathbb R$-basis of $N^1(X/U)$.  Numerical
	intersection defines an injective linear map
	\[
	\iota:N_1(X/U)\longrightarrow\mathbb R^\rho,
	\qquad
	\alpha\longmapsto(L_1\cdot\alpha,\ldots,L_\rho\cdot\alpha).
	\]
	Equation \eqref{eq:curve-index} shows that every normalized curve class lies
	in the discrete set
	\[
	\Lambda_{X/U}:=\{\alpha\in N_1(X/U)\mid
	L\cdot\alpha\in\mathbb Z\text{ for every Cartier divisor }L\text{ on }X\},
	\]
	whose image under $\iota$ is contained in $\mathbb Z^\rho$.
	
	Since the relative ample cone is open, there is an $\epsilon>0$ such that
	$H\pm\epsilon L_i$ is $U$-ample for every $i$.  Hence
	\[
	|L_i\cdot\alpha|\leq\epsilon^{-1}H\cdot\alpha
	\qquad\text{for every }\alpha\in\overline{NE}(X/U).
	\]
	Consequently, if 
	\[
	\beta\in \Theta:=\{\alpha\in\overline{NE}(X/U)\mid H\cdot\alpha\leq b\}
	\]
	then we have $|L_i\cdot\beta|\leq \frac{b}{\epsilon}$ for all $i$, hence the image of $\Theta$ under $\iota$ is contained in the box $[-\frac{b}{\epsilon},\frac{b}{\epsilon}]^\rho$.  Hence the above set is bounded, and since it is also closed, we conclude that it is compact. Therefore its intersection with the discrete set $\Lambda_{X/U}$ is finite.
\end{proof}

\begin{lemma}
\label{lem:cone-pushforward}
Let $f:Y\to X$ be a projective surjective morphism of integral schemes
projective over $U$.  Then
\[
 f_*\overline{NE}(Y/U)=\overline{NE}(X/U).
\]
\end{lemma}

\begin{proof}
This is the relative version of \cite[Corollary~3.22]{fulger_positive_cones} with $k=1$, and their proof goes through mostly unchanged.  To this end, use the relative version of
\cite[Proposition~3.21]{fulger_positive_cones}.  More precisely, we choose $U$-ample Cartier divisors $H_X$ and $H_Y$ on $X$ and $Y$,
respectively. \cite[Proposition~3.21]{fulger_positive_cones} and \cite[Corollary~3.16]{fulger_positive_cones} say that there is a constant $M>0$ such that for every effective
real one-cycle $Z$ on $X$ whose components are curves over $U$, there is an
effective real one-cycle $Z'$ on $Y$ satisfying
\[
 f_*Z'=Z,
 \qquad
 H_Y\cdot Z'\leq M H_X\cdot Z.
\]
 The proof of \cite[Proposition~3.21]{fulger_positive_cones} carries over after
replacing the complete-intersection multisection used there by the closure in
$Y$ of a closed point of the generic fibre of $f$, and measuring degrees with
$U$-ample divisors. 
The proof of \cite[Corollary~3.22]{fulger_positive_cones} now applies. 
\end{proof}

\subsection{LMMP inputs}

\subsubsection{Resolution of singularities}

The exact conjecture regarding resolution of singularities we need to assume is the same as in \cite[Theorem 4.1]{hacon_four}:

\begin{conjecture}\label{conj:res}
Let $X$ be a normal integral scheme quasi-projective over an excellent ring, which is of finite Krull dimension, admits a dualizing complex with no residue characteristic $2$, $3$ or $5$.  The log resolution conjecture holds for $X$ if every log pair
$(Y,\Delta_Y)$ whose underlying scheme $Y$ is birational to $X$ admits a log
resolution obtained
by a sequence of regular blowups along the non-simple normal crossing locus at each stage.
\end{conjecture}

\subsubsection{The cone theorem in dimension three}

We use the following form of the cone theorem in dimension at
most three.

\begin{theorem}
\label{thm:threefold-cone}
Let $(X,\Delta)$ be a $\mathbb{Q}$-factorial dlt pair of dimension at
most three, projective over $U$, with $\Delta$ a $\mathbb Q$-boundary.  Then
there are countably many $(K_X+\Delta)$-negative extremal rays $R_j$ such that
\begin{enumerate}
\item
\[
 \overline{NE}(X/U)=
 \overline{NE}(X/U)_{K_X+\Delta\geq0}+\sum_jR_j.
\]
\item The rays $R_j$ do not accumulate in
$\overline{NE}(X/U)_{K_X+\Delta<0}$.
\item Each $R_j$ is generated by an integral geometric rational image $C_j$
satisfying
\[
 -6d_{C_j}\leq (K_X+\Delta)\cdot C_j<0.
\]
\end{enumerate}
Moreover, for every $U$-ample $\mathbb{Q}$-divisor $A$, only finitely many of the
rays $R_j$ are $(K_X+\Delta+A)$-negative.
\end{theorem}

\begin{proof}
For $\dim X\leq2$, the assertions other than the
geometric rational image refinement follow from
\cite[Theorem~2.46]{BMPSTWW1}.   If the image of $X$ in $U$ has positive
dimension, all but the geometric rational image statement are \cite[Theorem~9.28]{BMPSTWW1}.  The proof for the latter from \cite[Proof of Theorem~2.9]{fever_dream} applies verbatim in the setting of \cite{BMPSTWW1}.
If the image is a closed point of
positive characteristic, the full assertion is
\cite[Theorem~2.9]{fever_dream}.  For a characteristic-zero closed point, it
is the classical cone theorem
\cite[Theorem~3.7]{kollar_birational_1998}, applied after base
change to an algebraic closure and descent.
\end{proof}

\subsubsection{Relative fourfold MMP and dlt modifications.}

The following is the version of the four dimensional MMP 
 needed below.

\begin{theorem}[{\cite[Theorem~4.1]{hacon_four}}]
\label{thm:relative-fourfold-mmp}
Let $Z$ be a normal integral $\mathbb Q$-factorial scheme projective over
$U$, and let
\[
 \pi\colon(X,\Delta)\longrightarrow Z
\]
be a projective birational morphism from a $\mathbb Q$-factorial dlt pair of
dimension four, with $\Delta$ a $\mathbb Q$-boundary.  If
\(
 \operatorname{Ex}(\pi)\subseteq\lfloor\Delta\rfloor
\)
and the log-resolution hypothesis of \autoref{conj:res} holds, then one can run
a $(K_X+\Delta)$-MMP over $Z$ which terminates with a minimal model.
\end{theorem}

\begin{proof}
While these results are stated only in the case of perfect fields in
\cite{hacon_four}, the proof goes through without change now that we have
access to the full basepoint free theorem for three dimensional excellent schemes, with 
\cite{fever_dream} supplying the missing case.  See
\cite[p.~2, paragraph following Theorem~1.2]{hacon_four}.
\end{proof}

  We will mainly need this for the following standard
consequence.

\begin{proposition}
\label{prop:dlt-model}
Let $(Z,\Delta)$ be a normal integral $\mathbb{Q}$-factorial pair of 
dimension four, projective over $U$, where $\Delta$ is a $\mathbb{Q}$-boundary,
and suppose that \autoref{conj:res} holds for $Z$.
Then there is a projective birational morphism $f:X\to Z$ such that, on
setting
\[
 \Delta_X:=f_*^{-1}\Delta+\operatorname{Ex}(f)_{\mathrm{red}},
\]
the pair $(X,\Delta_X)$ is $\mathbb{Q}$-factorial dlt,
$K_X+\Delta_X$ is nef over $Z$, and every $f$-exceptional prime divisor is a
component of $\lfloor\Delta_X\rfloor$.  Moreover, there is an effective
$f$-exceptional $\mathbb{Q}$-divisor $F$ such that
\[
 f^*(K_Z+\Delta)=K_X+\Delta_X+F
\]
and
\[
 f(\operatorname{Supp}F)\subseteq\operatorname{Nlc}(Z,\Delta),
\]
where $\operatorname{Nlc}(Z,\Delta)$ denotes the non-log-canonical locus.  In
particular, if $(Z,\Delta)$ is log canonical away from a closed subset $T$, then
$F$ is supported over $T$.  If $(Z,\Delta)$ is log canonical, then $F=0$.
\end{proposition}

\begin{proof}
Run the construction of
\cite[Corollary 4.8]{hacon_four}, using
\autoref{thm:relative-fourfold-mmp} and the log resolution assumption in place of standard coefficients.
the construction does not assume that $(Z,\Delta)$ is log canonical, and the additional conclusions follow immediately in the case that it is not.  
\end{proof}

\subsubsection{ACC for LCTs}

We also use the ascending chain condition (ACC) for log canonical thresholds (LCTs) at codimension three points of our fourfold, by applying the following:

\begin{theorem}[{\cite[Theorem 1.5]{fever_dream}}]
Let $\Lambda\subseteq[0,1]$ and $\Omega\subseteq\mathbb{R}_{\geq0}$ be DCC
sets.  Let $\mathfrak L_3(\Lambda)$ be the class of excellent log canonical
pairs $(X,\Delta)$ such that
\begin{enumerate}
\item $\dim X=3$;
\item the coefficients of $\Delta$ belong to $\Lambda$; and
\item no residue characteristic of $X$ is equal to $2$, $3$, or $5$.
\end{enumerate}
For $(X,\Delta)\in\mathfrak L_3(\Lambda)$, let
$\mathfrak G_\Omega(X)$ be the set of effective $\mathbb{Q}$-Cartier divisors
on $X$ whose coefficients belong to $\Omega$.  For
$G\in\mathfrak G_\Omega(X)$, set
\[
 \lct(X,\Delta;G)
 :=\sup\{t\in\mathbb{R}_{\geq0}\mid
             (X,\Delta+tG)\text{ is log canonical}\}.
\]
Then
\[
 \operatorname{LCT}_3(\Lambda,\Omega):=
 \left\{
  \lct(X,\Delta;G)
  \ \middle|\
  (X,\Delta)\in\mathfrak L_3(\Lambda),\ 
  G\in\mathfrak G_\Omega(X)
 \right\}
\]
satisfies the ascending chain condition; equivalently, it contains no infinite
strictly increasing sequence.
\end{theorem}

In particular, we apply ACC for LCTs to the following numbers:

\begin{definition}
Let $X$ be a normal fourfold, let $\Delta$ be a $\mathbb{Q}$-boundary such that
$(X,\Delta)$ is log canonical in codimension three, and let $S$ be a prime
$\mathbb{Q}$-Cartier divisor.  Define the codimension three log canonical
threshold as
\[
 \begin{split}
 c(X,\Delta;S):=\sup\{t\in[0,1]\mid {}&(X,\Delta+tS)\text{ is log canonical}\\
 &\text{in codimension three}\}.
\end{split}
\]
Here ``log canonical in codimension three'' means log canonical at every point
of codimension at most three.
\end{definition}

Under the assumption that a fixed log resolution of $(X,\Delta+S)$ exists, this number is the minimum of $1$ and
the finitely many rational bounds supplied by divisors whose centers on $X$
have codimension at most three; in particular, it is rational, and $(X,\Delta+cS)$ is log canonical at every point of codimension at most three and has at least one non-klt center of codimension at most three.
Its non-log-canonical locus has codimension at least four and hence, since
$X$ has dimension four, is a finite set of closed points. 
If $(X,\Delta)$ is klt, or more generally dlt with $S$ not contained in $\rddown{\Delta}$, then the same finite collection of inequalities also shows
that $c(X,\Delta;S)>0$.

\section{Finding curves in extremal rays}

Our proof of the cone theorem proceeds by finding a curve satisfying the required bound on normalized length in any $K_X+\Delta$-negative extremal ray.  We do this by pushing them down from higher models using the following lemma.

\begin{lemma}\label{lem:extremal-lift}
Let $f:Y\to X$ be a projective birational morphism of normal integral schemes
projective over $U$.  If $R$ is an extremal ray of $\overline{NE}(X/U)$, there is an
extremal ray $R_Y$ of $\overline{NE}(Y/U)$ such that $f_*R_Y=R$.
\end{lemma}

\begin{proof}

By \autoref{lem:cone-pushforward},
\(
 f_*\overline{NE}(Y/U)=\overline{NE}(X/U).
\)
Consequently
\[
 \mathcal F:=\overline{NE}(Y/U)\cap f_*^{-1}(R)
\]
is a closed face of $\overline{NE}(Y/U)$, and $f_*\mathcal F=R$.  Let $H$ be
a $U$-ample divisor on $Y$.  Since $N_1(Y/U)$ is finite-dimensional and $H$
is strictly positive on $\overline{NE}(Y/U)\setminus\{0\}$, the slice
\[
 \mathcal B:=\{\alpha\in\mathcal F\mid H\cdot\alpha=1\}
\]
is a compact convex base of the cone $\mathcal F$.  Let $A$ be a $U$-ample divisor on $X$
and define
\(
 \varphi(\alpha):=f^*A\cdot\alpha.
\)
The function $\varphi$ is nonnegative on $\mathcal B$.  It is positive at some
point of $\mathcal B$, because $f_*\mathcal F=R\ne0$.  Therefore it has a positive maximum, which is attained by compactness.

The maximizing locus
\[
 \mathcal B_m:=
 \{\alpha\in\mathcal B\mid\varphi(\alpha)=\max_{\beta\in\mathcal B}\varphi(\beta)\}
\]
is a nonempty compact exposed face of $\mathcal B$, and hence it has an extreme
point $\alpha_0$.  Since $\mathcal B_m$ is a face, $\alpha_0$ is also an
extreme point of $\mathcal B$.  Extreme points of the base $\mathcal B$
correspond to extremal rays of $\mathcal F$. Thus
\(
 R_Y:=\mathbb R_{\geq0}\alpha_0
\)
is an extremal ray of $\mathcal F$.  Since $\mathcal F$ is a face of
$\overline{NE}(Y/U)$, the ray $R_Y$ is also extremal in
$\overline{NE}(Y/U)$.  Finally,
$\varphi(\alpha_0)=m>0$, so $f_*\alpha_0\ne0$.  As $f_*\alpha_0\in R$, this
gives $f_*R_Y=R$.
\end{proof}

The following two lemmas allow us to find the required extremal curves by reducing to the threefold case.

\begin{lemma}\label{lem:negative-divisor}
Let $Y$ be a projective scheme over $U$, let $T\geq0$
be a $\mathbb{Q}$-Cartier divisor, and let $R$ be an extremal ray of
$\overline{NE}(Y/U)$.  If $T\cdot R<0$, then $R$ is contained in the image of
\[
 \overline{NE}(T^\nu/U)\longrightarrow\overline{NE}(Y/U),
\]
where $T^\nu$ is the normalization of the support of $T$.  
\end{lemma}

\begin{proof}
Choose a nonzero $z\in R$ and effective one-cycles $z_n\to z$.  Split $z_n$
as $u_n+v_n$, where all components of $u_n$ are contained in $\Supp(T)$ and no
component of $v_n$ is contained in $\Supp(T)$.  Their degrees with respect to a fixed
divisor which is ample over $U$ are bounded.  After passing to a subsequence, both parts
converge, say to $u$ and $v$.  Then $z=u+v$, so extremality gives
$u,v\in R$.  Since $T\cdot v\geq0$ and $T\cdot z<0$, the class $u$ is
nonzero and so spans $R$.  Lifting the cycles $u_n$ to $T^\nu$ and using the same degree bound
shows that $u$ is in the displayed image.
\end{proof}

\begin{lemma}
\label{lem:boundary-ray}
Let $(Y,\Delta)$ be a $\mathbb{Q}$-factorial dlt pair of dimension four,
projective over $U$, and let $R$ be a $(K_Y+\Delta)$-negative extremal ray of
$\overline{NE}(Y/U)$.  If
$T\subseteq\lfloor\Delta\rfloor$ is a component satisfying $T\cdot R<0$,
then $R$ is generated by an integral geometric rational image $C$ such that
\[
 -6d_C\leq (K_Y+\Delta)\cdot C<0.
\]
\end{lemma}

\begin{proof}
Let $j:T^\nu\to Y$ be the natural morphism.  By
\autoref{lem:negative-divisor}, the ray $R$ is contained in the image of
$j_*\overline{NE}(T^\nu/U)$.  Hence
\[
 \mathcal F_T:=\overline{NE}(T^\nu/U)\cap j_*^{-1}(R)
\]
is a closed face of $\overline{NE}(T^\nu/U)$ satisfying $j_*\mathcal F_T=R$.
Repeating the compactness and maximizing argument in the proof of
\autoref{lem:extremal-lift} produces an extremal ray $R_T$ of
$\overline{NE}(T^\nu/U)$ such that $j_*R_T=R$.
Adjunction gives a dlt pair $(T^\nu,\Delta_{T^\nu})$ of dimension three such that
\[
 (K_Y+\Delta)|_{T^\nu}=K_{T^\nu}+\Delta_{T^\nu}.
\]
  and hence $R_T$ is negative for
 $K_{T^\nu}+\Delta_{T^\nu}$.  Take a
crepant $\mathbb{Q}$-factorial dlt modification $(\widetilde T,\Delta_{\widetilde T})$
by 
\cite[Corollary 9.21]{BMPSTWW1}.  By \autoref{lem:extremal-lift}, there is a $(K_{\widetilde T}+\Delta_{\widetilde T})$-negative extremal
ray $R_{\widetilde{T}}\subset\overline{NE}(\widetilde T/U)$ mapping onto $R_T$.  The relative
three-dimensional cone theorem \autoref{thm:threefold-cone} gives an integral
geometric rational image $C_T$ spanning $ R_{\widetilde{T}}$ with
\[
 -6d_{C_T}\leq
 (K_{\widetilde T}+\Delta_{\widetilde T})\cdot C_T<0.
\]
Let $g:\widetilde T\to Y$, let $C=g(C_T)_{\mathrm{red}}$, and write
$g_*C_T=mC$.  Crepancy, adjunction, and the projection formula give
\[
 (K_Y+\Delta)\cdot C
 =\frac{1}{m}(K_{\widetilde T}+\Delta_{\widetilde T})\cdot C_T
 \geq-6\frac{d_{C_T}}{m}\geq-6d_C,
\]
where the last inequality is \autoref{lem:curve-index-pushforward}.  $C$ is a geometric rational image by composing with $C_T\to C$.  Since it spans $R$,
its intersection with $K_Y+\Delta$ is negative.
\end{proof}

In the next proposition, we find a curve in a suitably negative extremal ray by modifying the variety to apply \autoref{lem:boundary-ray}.

\begin{proposition}\label{prop:positive-threshold}
Let $(X,\Delta)$ be a $\mathbb{Q}$-factorial klt pair of dimension four,
projective over $U$, with $\Delta$ a $\mathbb{Q}$-boundary, and suppose that
\autoref{conj:res} holds for $X$.
Let $S$ be a prime divisor and let $R$ be an
extremal ray of $\overline{NE}(X/U)$ such that
\[
 (K_X+\Delta)\cdot R<0,
 \qquad S\cdot R<0.
\]
Then $R$ is generated by an integral geometric rational image $C$ satisfying
\[
 -6d_C\leq(K_X+\Delta)\cdot C<0.
\]
\end{proposition}

\begin{proof}
	We produce an inductive sequence of birational models of $X$, which will eventually allow us to find the required curve.
Put $(X_0,\Delta_0,S_0,a_0,R_0)=(X,\Delta,S,0,R)$.
For a fixed $i\geq1$, suppose that the construction has reached level $i-1$, with projective birational morphism
\[
 \rho_{i-1}:=\pi_1\circ\cdots\circ\pi_{i-1}:X_{i-1}\longrightarrow X,
 \qquad \rho_0:=\operatorname{id}_X.
\]
Here $X_{i-1}$ is normal, integral and $\mathbb{Q}$-factorial, and
\[
 \Delta_{i-1}=(\rho_{i-1})_*^{-1}\Delta
      +\operatorname{Ex}(\rho_{i-1})_{\mathrm{red}},
 \qquad
 S_{i-1}=(\rho_{i-1})_*^{-1}S.
\]
Thus $\Delta_{i-1}$ is a $\mathbb{Q}$-boundary and $S_{i-1}$ is a prime
divisor. The rational number $a_{i-1}$ for $i>1$ is defined at the previous level by $a_{i-1}=c(X_{i-2},\Delta_{i-2};S_{i-2})$.  The pair
$(X_{i-1},\Delta_{i-1}+a_{i-1}S_{i-1})$ is dlt, and
$R_{i-1}\subseteq\overline{NE}(X_{i-1}/U)$ is an extremal ray satisfying
\[
 (\rho_{i-1})_*R_{i-1}=R,
 \qquad
 (K_{X_{i-1}}+\Delta_{i-1}+a_{i-1}S_{i-1})\cdot R_{i-1}<0,
 \qquad
 S_{i-1}\cdot R_{i-1}<0.
\]
The construction would stop if $R_{i-1}$ were negative for any component of $\rddown{\Delta_{i-1}+a_{i-1}S_{i-1}}$.  If it has not stopped, we must have $\operatorname{coeff}_{S_{i-1}}(\Delta_{i-1})+a_{i-1}<1$.

Assuming the construction has not terminated at level $i-1$, define
\[
 a_i:=c(X_{i-1},\Delta_{i-1};S_{i-1}).
\]
We claim that $a_i>a_{i-1}$.  For $i=1$, positivity of the codimension-three
log canonical threshold for the klt pair $(X_0,\Delta_0)$ gives
$a_1>0=a_0$.  Suppose that $i>1$.  By induction, $a_{i-1}>0$.  Ignore the finitely many log canonical
strata of $(X_{i-1},\Delta_{i-1}+a_{i-1}S_{i-1})$ which are codimension four closed points, so that every remaining log canonical center is a coefficient-one
stratum of codimension at most three.  Since $\operatorname{coeff}_{S_{i-1}}(\Delta_{i-1})+a_{i-1}<1$, we see that
$S_{i-1}$ cannot contain any such stratum.  The finitely many inequalities on
a fixed log resolution therefore remain valid after increasing the coefficient
of $S_{i-1}$ by some $\delta>0$.  Hence
\[a_i\geq a_{i-1}+\delta>a_{i-1}.\]

By construction,
$(X_{i-1},\Delta_{i-1}+a_iS_{i-1})$ is log canonical away from finitely many
closed points.  Apply \autoref{prop:dlt-model} to this pair and write
\[
 \pi_i:X_i\longrightarrow X_{i-1},\qquad
 \Delta_i:=(\pi_i)_*^{-1}\Delta_{i-1}
       +\operatorname{Ex}(\pi_i)_{\mathrm{red}},
 \qquad
 S_i:=(\pi_i)_*^{-1}S_{i-1}.
\]
Then
\begin{equation}\label{eq:iterated-threshold}
 \pi_i^*(K_{X_{i-1}}+\Delta_{i-1}+a_iS_{i-1})
 =K_{X_i}+\Delta_i+a_iS_i+F_i,
\end{equation}
where $(X_i,\Delta_i+a_iS_i)$ is $\mathbb{Q}$-factorial dlt,
 every
$\pi_i$-exceptional divisor is a component of $\lfloor\Delta_i\rfloor$, and
$F_i\geq0$ is exceptional and supported over finitely many closed points of $X_{i-1}$.

By \autoref{lem:extremal-lift}, choose an extremal ray $R_i$ mapping onto
$R_{i-1}$.  Since $a_i>a_{i-1}$ and $S_{i-1}\cdot R_{i-1}<0$, increasing the
coefficient of $S_{i-1}$ makes the adjoint divisor more negative on
$R_{i-1}$.  Moreover, $F_i\cdot R_i\geq0$.
Indeed, if $F_i\cdot R_i<0$, then \autoref{lem:negative-divisor} implies that
    $R_i$ is contained in the image of
    $\overline{NE}(F_i^\nu/U)\to\overline{NE}(X_i/U)$.  Since $F_i$ is supported over
finitely many closed points, $\pi_{i*}$ vanishes on this image.  This would give
$\pi_{i*}R_i=0$, contradicting $\pi_{i*}R_i=R_{i-1}$.  It follows from
\eqref{eq:iterated-threshold} that
\[
 (K_{X_i}+\Delta_i+a_iS_i)\cdot R_i<0.
\]

Write
\[
 \pi_i^*S_{i-1}=S_i+G_i,
 \qquad G_i\geq0,
 \qquad \operatorname{Supp}(G_i)\subseteq\lfloor\Delta_i\rfloor.
\]
We know that $\pi_i^*S_{i-1}$ is negative on $R_i$.  If a component of $G_i$ is negative on
$R_i$, we stop.  If no such component exists, then $S_i\cdot R_i<0$.  If
$S_i$ is a component of $\lfloor\Delta_i+a_iS_i\rfloor$, we stop again.
Otherwise
\(
 \operatorname{coeff}_{S_i}(\Delta_i)+a_i<1,
\)
and the construction has reached level $i$ without terminating.

We can now iterate this construction, which only terminates if it finds a coefficient one component of the boundary on which $R$ is negative.  If it does not stop, it gives a strictly increasing sequence
\[
 0<a_1<a_2<a_3<\cdots<1.
\]
The coefficients of all the $\Delta_i$ belong to the fixed finite set consisting
of the coefficients of $\Delta_0$ together with $1$.  For $a_i<1$, attainment on
the fixed log resolution forces
$a_i=c(X_{i-1},\Delta_{i-1};S_{i-1})$ to be computed at a center of codimension at
most three.  Localizing at a codimension three point contained in such a center gives a
log canonical threshold of an excellent singularity of dimension 
three.  The ACC for three-dimensional log canonical thresholds
\cite[Theorem 1.5]{fever_dream} rules out this sequence.

Thus, after finitely many steps, there is a dlt pair
$(X_N,\Delta_N+a_NS_N)$, an extremal ray $R_N$, and a component
$T\subseteq\lfloor\Delta_N+a_NS_N\rfloor$ such that both
\[
 (K_{X_N}+\Delta_N+a_NS_N)\cdot R_N<0,
 \qquad T\cdot R_N<0.
\]
By \autoref{lem:boundary-ray}, $R_N$ is generated by an integral geometric
rational image $C_N$
with
\[
 -6d_{C_N}\leq(K_{X_N}+\Delta_N+a_NS_N)\cdot C_N<0.
\]

Let $C_i$ be the reduced image of $C_N$ on $X_i$, and let $m_i$ be the degree
of $C_i\to C_{i-1}$.  None of these curves are contracted, because they each span their extremal ray, whose image is not zero.  Since $F_i\cdot C_i\geq0$, $S_{i-1}\cdot C_{i-1}<0$, and
$a_i\geq a_{i-1}$, the projection formula gives
\begin{align*}
 &(K_{X_{i-1}}+\Delta_{i-1}+a_{i-1}S_{i-1})\cdot C_{i-1}\\
 &\quad\geq
 (K_{X_{i-1}}+\Delta_{i-1}+a_iS_{i-1})\cdot C_{i-1}\\
 &\quad=\frac{1}{m_i}
 (K_{X_i}+\Delta_i+a_iS_i+F_i)\cdot C_i\\
 &\quad\geq\frac{1}{m_i}
 (K_{X_i}+\Delta_i+a_iS_i)\cdot C_i.
\end{align*}
By induction we may assume that 
\begin{align*}
(K_{X_i}+\Delta_i+a_iS_i)\cdot C_i&\geq -6d_{C_i} 
\end{align*}
Hence the above calculation together with \autoref{lem:curve-index-pushforward} gives
\[
	(K_{X_{i-1}}+\Delta_{i-1}+a_{i-1}S_{i-1})\cdot C_{i-1}\geq \frac{-6 d_{C_i}}{m_i}\geq -6d_{C_{i-1}}
\]

Descending inductively therefore gives
\[
 -6d_{C_0}\leq(K_X+\Delta)\cdot C_0<0.
\]
The curve $C_0$ spans $R$ and is a geometric rational image by taking the image of $C_{N}\to C_0$, as required.
\end{proof}

\section{Proof of the cone theorem}

\begin{proof}[Proof of \autoref{thm:main-cone}]
Fix a $(K_X+\Delta)$-negative extremal ray $R$.
Take a crepant $\mathbb{Q}$-factorial dlt modification $f:Y\to X$, with
\[
 f^*(K_X+\Delta)=K_Y+\Delta_Y,
\]
and use \autoref{lem:extremal-lift} to choose an extremal ray $R_Y$ mapping
onto $R$.  It suffices to find an integral geometric rational image $C_Y$
spanning $R_Y$ with
\[
 -6d_{C_Y}\leq(K_Y+\Delta_Y)\cdot C_Y<0.
\]
Indeed, $C_Y$ is not contracted by $f$.  If $C$ is its reduced image and $m$
is the degree of $C_Y\to C$, then crepancy and the projection formula give
\[
 (K_X+\Delta)\cdot C=\frac1m(K_Y+\Delta_Y)\cdot C_Y
 \geq-6\frac{d_{C_Y}}m\geq-6d_C.
\]
Thus  $C$ is a geometric
rational image spanning $R$ and satisfying the required bound.
For the rest of the argument, relabel
$(Y,\Delta_Y,f^*M,R_Y)$ as $(X,\Delta,M,R)$.  Thus $(X,\Delta)$ is
$\mathbb{Q}$-factorial dlt.

If some component $T\subseteq\lfloor\Delta\rfloor$ is negative on $R$,
\autoref{lem:boundary-ray} gives the required curve.  We may thus
assume that every component of $\lfloor\Delta\rfloor$ is nonnegative on $R$.
For a sufficiently small rational $\epsilon>0$, put
\(
 \Delta_\epsilon=\Delta-\epsilon\lfloor\Delta\rfloor.
\)
The pair $(X,\Delta_\epsilon)$ is klt and its adjoint divisor is still negative
on $R$.  Since $M$ is effective and $M\cdot R<0$, some prime component $S$ of
$M$ is negative on $R$.  Applying \autoref{prop:positive-threshold} to the klt
pair $(X,\Delta_\epsilon)$ produces an integral geometric rational image $C$
spanning $R$ such that
\[
 -6d_C\leq(K_X+\Delta_\epsilon)\cdot C<0.
\]
Since every component of $\lfloor\Delta\rfloor$ is nonnegative on $R$,
\[
 (K_X+\Delta)\cdot C
 =(K_X+\Delta_\epsilon)\cdot C
   +\epsilon\lfloor\Delta\rfloor\cdot C
 \geq-6d_C.
\]
On the other hand, $(K_X+\Delta)\cdot C<0$ because $C$ spans $R$. 

 The cone theorem is now a formal consequence of the existence of a curve of
 bounded normalized length in every $K_X+\Delta$-negative extremal ray.  Fix a
 $U$-ample Cartier divisor $H$.  The slice
 \[
  \mathcal P_H:=
  \{z\in\overline{NE}(X/U)\mid H\cdot z=1\}
 \]
 is compact by the argument in
 \autoref{lem:normalized-classes-discrete}.  Its extreme points are precisely
 the intersections of the extremal rays of $\overline{NE}(X/U)$ with
 $\mathcal P_H$.  By the finite-dimensional
 Minkowski--Carath\'eodory theorem, every point of $\mathcal P_H$ is a finite
 convex combination of extreme points.  Grouping the terms whose rays are
 $(K_X+\Delta)$-nonnegative and those whose rays are
 $(K_X+\Delta)$-negative therefore gives
\[
 \overline{NE}(X/U)=
 \overline{NE}(X/U)_{K_X+\Delta\geq0}+\sum_i R_i,
\]
 where the $R_i$ are all the $(K_X+\Delta)$-negative extremal rays.  We have just shown that in each $R_i$ is an integral geometric rational
 image $C_i$ satisfying the stated bound.  The normalized classes
 $\langle C_i\rangle$ belong to the countable discrete set
 $\Lambda_{X/U}$ defined in
 the proof of \autoref{lem:normalized-classes-discrete}.  Since distinct rays
 give distinct normalized classes, there are only countably many such rays.

 We next verify the local finiteness assertion.  Let $A$ be a $U$-ample
 $\mathbb{Q}$-divisor and choose a positive integer $m$ such that $mA$ is a
 $U$-ample Cartier divisor.  If $R_i=\mathbb{R}_{\geq0}[C_i]$ is
 $(K_X+\Delta+A)$-negative, then
\[
  0<(mA)\cdot\langle C_i\rangle
  <-m(K_X+\Delta)\cdot\langle C_i\rangle\leq6m.
\]
 By \autoref{lem:normalized-classes-discrete}, only finitely many normalized
 classes, and hence only finitely many of the rays $R_i$, satisfy this
 inequality.

 Finally, we prove the non-accumulation.  If distinct $R_{i_j}$
 accumulated at a ray $R$ contained in the open half-space
 $(K_X+\Delta)<0$, and if $z_j\in R_{i_j}\cap\mathcal P_H$ converged to
 $z\in R\cap\mathcal P_H$, choose a rational number
\[
  0<\delta<-(K_X+\Delta)\cdot z.
\]
 Then $(K_X+\Delta+\delta H)\cdot z_j<0$ for all sufficiently large $j$.
 This contradicts the finiteness just proved, applied to the $U$-ample
 $\mathbb{Q}$-divisor $\delta H$.  Thus the $R_i$ do not accumulate in
 $\overline{NE}(X/U)_{K_X+\Delta<0}$, and all the assertions follow.
\end{proof}

\begin{proof}[Proof of \autoref{cor:rationality}]
	Let $C_i$ be the curves supplied by \autoref{thm:main-cone}.  The cone
	decomposition gives
	\[
	r=\inf_i\frac{H\cdot C_i}{-(K_X+\Delta)\cdot C_i}.
	\]
	Set
	\[
	n_i:=\frac{H\cdot C_i}{d_{C_i}},
	\qquad
	m_i:=\frac{-a(K_X+\Delta)\cdot C_i}{d_{C_i}}.
	\]
	By \eqref{eq:curve-index}, $n_i$ and $m_i$ are positive integers, while the
	length bound gives $m_i\leq6a$.  Hence
	\[
	\frac{H\cdot C_i}{-(K_X+\Delta)\cdot C_i}
	=\frac{a n_i}{m_i}.
	\]
	Since $m_i$ belongs to the finite set $\{1,\ldots,6a\}$ and $n_i$ is a
	positive integer, the displayed infimum is attained by one of the $C_i$.
\end{proof}

\begin{proof}[Proof of \autoref{cor:pl-contraction}]
	\autoref{thm:main-cone} verifies the hypotheses of \cite[Proposition 4.5]{hacon_four} in our situation.
	\end{proof}

\bibliographystyle{acm}
\bibliography{Library}

\end{document}